\documentclass[11pt]{article}
\usepackage[utf8]{inputenc}
\usepackage{fullpage}
\usepackage[numbers]{natbib}
\usepackage{hyperref}
\usepackage{amsmath,cleveref,autonum} %
\usepackage{amsthm,amsfonts,amsbsy,amssymb,commath,dsfont}
\usepackage{color}
\usepackage{xspace}
\usepackage[affil-sl]{authblk}
\usepackage{graphicx}

\makeatletter
\autonum@generatePatchedReferenceCSL{Cref}
\makeatother

\newcommand\lo[1]{\backslash #1}

\usepackage{amsmath,amsbsy,amsfonts,amssymb,amsthm,dsfont,mleftright,commath}

\def\ddefloop#1{\ifx\ddefloop#1\else\ddef{#1}\expandafter\ddefloop\fi}

\def\ddef#1{\expandafter\def\csname bf#1\endcsname{\ensuremath{\mathbf{#1}}}}
\ddefloop ABCDEFGHIJKLMNOPQRSTUVWXYZabcdefghijklmnopqrstuvwxyz\ddefloop

\def\ddef#1{\expandafter\def\csname bf#1\endcsname{\ensuremath{\pmb{\csname #1\endcsname}}}}
\ddefloop {alpha}{beta}{gamma}{delta}{epsilon}{varepsilon}{zeta}{eta}{theta}{vartheta}{iota}{kappa}{lambda}{mu}{nu}{xi}{pi}{varpi}{rho}{varrho}{sigma}{varsigma}{tau}{upsilon}{phi}{varphi}{chi}{psi}{omega}{Gamma}{Delta}{Theta}{Lambda}{Xi}{Pi}{Sigma}{varSigma}{Upsilon}{Phi}{Psi}{Omega}{ell}\ddefloop

\def\ddef#1{\expandafter\def\csname bb#1\endcsname{\ensuremath{\mathbb{#1}}}}
\ddefloop ABCDEFGHIJKLMNOPQRSTUVWXYZ\ddefloop

\def\ddef#1{\expandafter\def\csname c#1\endcsname{\ensuremath{\mathcal{#1}}}}
\ddefloop ABCDEFGHIJKLMNOPQRSTUVWXYZ\ddefloop

\def\ddef#1{\expandafter\def\csname v#1\endcsname{\ensuremath{\boldsymbol{#1}}}}
\ddefloop ABCDEFGHIJKLMNOPQRSTUVWXYZabcdefghijklmnopqrstuvwxyz\ddefloop

\def\ddef#1{\expandafter\def\csname bv#1\endcsname{\ensuremath{\bar{\boldsymbol{#1}}}}}
\ddefloop ABCDEFGHIJKLMNOPQRSTUVWXYZabcdefghijklmnopqrstuvwxyz\ddefloop

\def\ddef#1{\expandafter\def\csname hv#1\endcsname{\ensuremath{\hat{\boldsymbol{#1}}}}}
\ddefloop ABCDEFGHIJKLMNOPQRSTUVWXYZabcdefghijklmnopqrstuvwxyz\ddefloop

\def\ddef#1{\expandafter\def\csname tv#1\endcsname{\ensuremath{\tilde{\boldsymbol{#1}}}}}
\ddefloop ABCDEFGHIJKLMNOPQRSTUVWXYZabcdefghijklmnopqrstuvwxyz\ddefloop

\def\ddef#1{\expandafter\def\csname v#1\endcsname{\ensuremath{\boldsymbol{\csname #1\endcsname}}}}
\ddefloop {alpha}{beta}{gamma}{delta}{epsilon}{varepsilon}{zeta}{eta}{theta}{vartheta}{iota}{kappa}{lambda}{mu}{nu}{xi}{pi}{varpi}{rho}{varrho}{sigma}{varsigma}{tau}{upsilon}{phi}{varphi}{chi}{psi}{omega}{Gamma}{Delta}{Theta}{Lambda}{Xi}{Pi}{Sigma}{varSigma}{Upsilon}{Phi}{Psi}{Omega}{ell}\ddefloop

\newcommand\T{{\scriptscriptstyle{\mathsf{T}}}}

\renewcommand\v{\ensuremath{\boldsymbol}}

\newtheorem{theorem}{Theorem}
\newtheorem{lemma}{Lemma}

\newcommand\diag[2][0]{\operatorname{diag}\del[#1]{#2}}
\newcommand\tr[2][0]{\operatorname{tr}\del[#1]{#2}}

\newcommand\compl{{\operatorname{c}}}
\newcommand\independent{\protect\mathpalette{\protect\independenT}{\perp}}
\def\independenT#1#2{\mathrel{\rlap{$#1#2$}\mkern2mu{#1#2}}}
\newcommand\Normal{{\operatorname{N}}}
\newcommand\op{{\operatorname{op}}}
\newcommand\eigmin{{\operatorname{\lambda_{\min}}}}

\newcommand\ith{$i^{\text{th}}$\xspace}
\newcommand\jth{$j^{\text{th}}$\xspace}

\newcommand\ind[1]{\mathds{1}_{\{#1\}}}

\title{On the proliferation of support vectors in high dimensions}
\author[1,2]{Daniel Hsu}
\author[3,4]{Vidya Muthukumar}
\author[1]{Ji Xu}
\affil[1]{Department of Computer Science, Columbia University}
\affil[2]{Data Science Institute, Columbia University}
\affil[3]{School of Electrical and Computer Engineering, Georgia Institute of Technology}
\affil[4]{School of Industrial and Systems Engineering, Georgia Institute of Technology}

\begin{document}
\maketitle

\begin{abstract}

The support vector machine (SVM) is a well-established classification method whose name refers to the particular training examples, called support vectors, that determine the maximum margin separating hyperplane.
The SVM classifier is known to enjoy good generalization properties when the number of support vectors is small compared to the number of training examples.
However, recent research has shown that in sufficiently high-dimensional linear classification problems, the SVM can generalize well despite a \emph{proliferation of support vectors} where all training examples are support vectors.
In this paper, we identify new deterministic equivalences for this phenomenon of support vector proliferation, and use them to (1) substantially broaden the conditions under which the phenomenon occurs in high-dimensional settings, and (2) prove a nearly matching converse result.

\end{abstract}

\section{Introduction}
\label{sec:intro}

The Support Vector Machine (SVM) is one of the most well-known and commonly used methods for binary classification in machine learning~\citep{vapnik1982estimation,cortes1995support}.
Its homogeneous version in the linearly separable setting (commonly also known as the \emph{hard-margin SVM}) is defined as the solution to an optimization problem characterizing the linear classifier (a separating hyperplane) that maximizes the minimum margin achieved on the $n$ training examples $(\vx_1,y_1),\dotsc,(\vx_n,y_n) \in \bbR^d \times \{-1,+1\}$:
\begin{align}
  \max_{\vw \in \bbR^d, \gamma \geq 0} \quad & \gamma
  \qquad \text{subject\ to} \quad \operatorname{margin}_i(\vw) \geq \gamma \quad \text{for all $i=1,\dotsc,n$} ,
  \label{eq:svm}
\end{align}
where
\begin{align}
  \operatorname{margin}_i(\vw)
  & := \begin{cases}
    y_i\vx_i^\T\vw/\|\vw\|_2 & \text{if $\vw \neq \v0$} \\
    0 & \text{if $\vw = \v0$}
  \end{cases}
\end{align}
is the margin achieved by $\vw$ on the \ith training example\footnote{We only consider homogeneous linear classifiers in this paper and hence have omitted the bias term. The equivalent, but more standard, form of this problem is presented as \Cref{eq:svm-primal} in \Cref{sec:svm-dual}.} $(\vx_i,y_i)$.
The SVM gets its name from the fact that the solution $(\vw^\star,\gamma^\star)$ depends only on the set of training examples that achieve the minimum margin value, $\gamma^\star$.
These examples are known as the ``support vectors'', and it is well-known that the weight vector $\vw^\star$ can be written as a (non-negative) linear combination of the $y_i\vx_i$ corresponding to support vectors.
More precisely, the dual form of the solution expresses the weight vector $\vw^\star = \sum_{i=1}^n \alpha_i^\star y_i\vx_i$ in terms of dual variables $\alpha_1^\star,\dotsc,\alpha_n^\star \geq 0$.
This constitutes a concise representation of the solution---just the list of non-zero dual variables $\alpha_i^\star$ and corresponding data points.
This remarkable property of the SVM is particularly important in its ``kernelized'' extension~\citep{boser1992training,smola2002learning}, where the dimension $d$ may be very large (or, in fact, infinite) but inner products can be computed efficiently.

The number of support vectors, if sufficiently small, has interesting consequences for the generalization error of the hard-margin SVM solution.
Techniques based on leave-one-out analysis and sample compression~\cite{vapnik1995nature,graepel2005pac,germain2011pac} bound the generalization error as a linear function of the fraction of support vectors and have no explicit dependence on the dimension $d$.
In particular, if the number of support vectors can be shown to be $o(n)$ with high probability, these bounds imply ``good generalization" of the SVM solution in the sense that the generalization error of the SVM is upper-bounded by a quantity that tends to zero as $n \to \infty$. %
Moreover, this sparsity in support vectors can be demonstrated in sufficiently low-dimensional settings using asymptotic arguments~\cite{dietrich1999statistical,buhot2001robust,malzahn2005statistical}.
However, the story is starkly different in the modern high-dimensional (also called \emph{overparameterized}) regime; in fact, quite the opposite can happen.
Recent work comparing classification and regression tasks under the high-dimensional linear model~\citep{muthukumar2020classification} showed that under sufficient ``effective overparameterization", e.g., $d \sim n \log n$ under isotropic Gaussian design, \emph{every training example is a support vector with high probability}.
That is, the fraction of support vectors is exactly $1$ with high probability.
This establishes a remarkable link between the SVM and solutions that interpolate training data, allowing an entirely different set of recently developed techniques that analyze interpolating solutions in regression tasks~\cite{belkin2019two,bartlett2020benign,hastie2019surprises,mei2019generalization,mitra2019understanding,muthukumar2020harmless} to be applied to the SVM.
Using this equivalence,~\citet{muthukumar2020classification} showed the existence of intermediate levels of overparameterization in which all training examples are support vectors with high probability, but the ensuing SVM solution still generalizes well.
This characterization was derived for a specific overparameterized ensemble inspired by spiked covariance models~\cite{wang2017asymptotics,mahdaviyeh2019risk}.
More importantly, the level of overparameterization considered there was only sufficiently, not necessarily, high enough for support vector proliferation.

In this paper, we establish necessary and sufficient conditions for the phenomenon of support vector proliferation to occur with high probability for a range of high-dimensional linear ensembles, including sub-Gaussian and Haar design of the covariate matrix.
In other words, for sufficiently high \emph{effective overparameterization} (measured through quantities that are related to effective ranks of the covariance matrix as identified by~\citet{bartlett2020benign}), we show that all training examples are support vectors with high probability.
We also provide a weak converse: in the absence of a certain level of overparameterization, at least one training example is not a support vector with constant probability.

\subsection*{Related work}
The number of support vectors has been previously studied in several contexts on account of the aforementioned connection to generalization error both in classical regimes using sample compression bounds~\cite{vapnik1995nature,graepel2005pac,germain2011pac}, and the modern high-dimensional regime~\cite{muthukumar2020classification,chatterji2020finite}.
Several works investigate the thermodynamic limit where both the dimension of the input data $d$ and the number of training data $n$ both tend to infinity at a fixed ratio $\delta = n/d$~\citep[e.g.,][]{dietrich1999statistical,buhot2001robust,malzahn2005statistical,liu2019exact}.
One particular result of note is that of \citet{buhot2001robust}, who consider a linearly\footnote{We note that the main interest of \citeauthor{buhot2001robust} is in SVMs with non-linear feature maps; we quote one of their results specialized to the linear setting.} separable setting where the training data inputs are drawn iid from a $d$-dimensional \emph{isotropic} normal distribution.
They find that the typical fraction of training examples that are support vectors approaches the following (in the limit as both $n,d \to \infty$):
\begin{equation}
  \begin{cases}
    \frac{0.952}{\delta} & \text{for $\delta \gg 1$} , \\
    1 - \sqrt{\frac{2\delta}{\pi}} \exp\del{-\frac{1}{2\delta}}
                         & \text{for $\delta \ll 1$} .
  \end{cases}
\end{equation}
In the classical regime, where $n \gg d$ (i.e., $\delta \gg 1$), a combination of this asymptotic estimate with sample compression arguments yields generalization error bounds of order $O(1/\delta) = O(d/n)$, which tend to zero as $\delta \to \infty$.
However, in the high-dimensional regime, where $d \gg n$ (i.e., $\delta \ll 1$), the fraction of examples that are support vectors quickly approaches $1$ as $\delta \to 0$.
In these cases, the generalization error bounds based on support vectors no longer provide non-trivial guarantees.

\citet{muthukumar2020classification} recently provided a non-asymptotic result for this isotropic case considered above.
They found that if $d$ grows somewhat faster than $n$ (specifically, $d \sim n \log n$), then the fraction of examples that are support vectors is $1$ with very high probability.
They also showed that the fraction of support vectors obtained by the hard-margin linear SVM can tend to $1$ in anisotropic settings \emph{if} the setting is sufficiently high-dimensional; this is captured by notions of effective rank of the covariance matrix of the linear featurizations~\cite{bartlett2020benign}.
Our results greatly sharpen the \emph{sufficient} conditions provided there; see \Cref{sec:results} for a detailed comparison, and in particular, \Cref{sec:discussion} for additional discussion of implications for generalization error bounds.

\citet{chatterji2020finite} also recently showed that the SVM can generalize well in overparameterized regimes.
In their work, the data are generated by a linear model inspired by Fisher's linear discriminant analysis, and establish their results under the assumption of sufficiently high separation between the means of the two classes.
Their results are based on a direct analysis of the SVM, but do not make any claims about the number of support vectors.

The number of support vectors has also been studied in non-separable but low-dimensional settings, using suitable variants of the SVM optimization problem.
These variants include the \emph{soft-margin SVM}~\citep{cortes1995support} and the \emph{$\nu$-SVM}~\citep{scholkopf2000new}.
In both of these, the hard-margin constraint is relaxed and support vectors include training examples that are exactly on the margin as well as \emph{margin violations}.
The soft-margin SVM doees this by introducing slack variables in the margin constraints on examples, and uses a hyper-parameter to control the trade-off between the margin maximization objective and the sum of constraint violations.
The $\nu$-SVM provides somewhat more direct control on the number of support vectors: the hyper-parameter $\nu$ is an upper-bound on the fraction of margin violations and a lower-bound on the fraction of all support vector examples.
First, for a suitable choice of the hyper-parameter, the fraction of examples that are support vectors in the soft-margin SVM can be related to the Bayes error rate when certain kernel functions are used~\citep{steinwart2003sparseness,bartlett2007sparseness}.
Indeed, this fact has motivated algorithmic developments for sparsifying the SVM solution~\citep[e.g.,][]{burges1996simplified,downs2001exact,keerthi2006building}.
Second, under some general conditions on the data distribution, it is also shown for the $\nu$-SVM~\citep[Proposition 5]{scholkopf2000new} that as $n\to\infty$ for a fixed dimension $d$, all support vectors are of the margin violation category.
These results for non-separable but low-dimensional settings are not directly comparable to ours, which hold in the high-dimensional (therefore, typically separable) regime.
Notably, our results on the support vector proliferation do not require the presence of label noise---i.e., the Bayes error rate can be zero and still, every example may be a support vector.

In addition to the aforementioned sample compression bounds that explicitly use the number of support vectors, there is a distinct line of work on generalization error of SVMs based on the margin $\gamma$ achieved on the training examples~\citep{bartlett1999generalization,zhang2002covering,bartlett2002rademacher,mcallester2003simplified,gronlund2020near}.
However, in the settings we consider, these generalization error bounds are never smaller than a universal constant (e.g., $1/\sqrt2$), as pointed out by \citet[Section 6]{muthukumar2020classification} and expanded upon in \Cref{sec:discussion}.
It is worth mentioning that the margin-based bounds, as well as the bounds based on the number of support vectors, make no (or very few) assumptions about the distribution of the training examples. 
The distribution-free quality makes the bounds widely applicable, but it also limits their ability to capture certain generalization phenomena, such as those from \citep{muthukumar2020classification,chatterji2020finite}.

Our work bears some resemblance to the early work of \citet{cover1965geometrical} on linear classification.
There, the concern is the number of independent features necessary and sufficient for a data set (with fixed, non-random labels) to become linear separable.
Linear separability just requires the existence of $\vw \in \bbR^d$ such that $\operatorname{margin}_i(\vw)>0$ for all $i=1,\dotsc,n$, but these margin values could vary across examples.
In contrast, our work considers necessary and sufficient conditions under which the margins achieved are all the same maximum (positive) value.

There have been several developments on support vector proliferation since the initial publication of our work.
First, \citet{ardeshir2021support} strengthened our converse result under the independent features model.
For isotropic features, they show that $d = \Omega(n \log n)$ is necessary for a constant probability of support vector proliferation.
They also provide a converse result for anisotropic features in terms of effective dimensions.
In the special case of standard Gaussian features, they find that the transition occurs around $d \sim 2 n \log n$, and they bound the width of the transition.
\citeauthor{ardeshir2021support} also give empirical evidence for the universality of support vector proliferation under broader classes of feature distributions.
Independently,~\citet{wang2020benign} showed that support vector proliferation also occurs under sufficient effective overparameterization under the Gaussian mixture model;~\citet{cao2021risk} further sharpened and generalized these results to general sub-Gaussian mixture models.
Finally,~\citet{wang2021benign} considered the multiclass case and showed a high-probability equivalence between not only the one-vs-all SVM~\cite{rifkin2004defense} and interpolation (which follows as a direct consequence of this work), but also the multiclass SVM~\cite{weston1998multi}.
While the dual of the multiclass SVM required a different and novel treatment, subsequent steps in their proof leverage, in part, the arguments that are provided in this paper.

\section{Setting}
\label{sec:setting}

In this section, we introduce notation for the SVM problem, and describe the probabilistic models of the training data under which we conduct our analysis.

\subsection{SVM optimization problem}
\label{sec:svm-dual}

Our analysis considers the standard setting for homogeneous binary linear classification with SVMs.
In this setting, one has $n$ training examples $(\vx_1,y_1), \dotsc, (\vx_n,y_n) \in \bbR^d \times \{-1,+1\}$.
A homogeneous linear classifier is specified by a weight vector $\vw \in \bbR^d$, so that the prediction of this classifier on $\vx \in \bbR^d$ is given by the sign of $\vx^\T\vw$.
The ambiguity of the sign when $\vx^\T\vw = 0$ is not important in our analysis.

The SVM optimization problem from \Cref{eq:svm} is more commonly written as
\begin{equation}
  \begin{aligned}
    \min_{\vw \in \bbR^d} \quad & \frac12 \|\vw\|_2^2 \\
    \text{subj.\ to} \quad & y_i\vx_i^\T\vw \geq 1 \quad \text{for all $i=1,\dotsc,n$} .
  \end{aligned}
  \label{eq:svm-primal}
\end{equation}
The well-known Lagrangian dual of \Cref{eq:svm-primal} can be written entirely in terms of the vector of labels $\vy := (y_1,\dotsc,y_n) \in \bbR^n$ and the $n \times n$ Gram (or kernel) matrix $\vK$ corresponding to $\vx_1,\dotsc,\vx_n$, i.e., $K_{i,j} := \vx_i^\T\vx_j$ for all $1 \leq i, j \leq n$:
\begin{equation}
  \begin{aligned}
    \max_{\valpha \in \bbR^n} \quad & \sum_{i=1}^n \alpha_i - \frac12 \valpha^\T \diag{\vy}^\T \vK \diag{\vy} \valpha \\
    \text{subj.\ to} \quad & \alpha_i \geq 0 \quad \text{for all $i=1,\dotsc,n$} .
  \end{aligned}
  \label{eq:svm-dual}
\end{equation}
Above, we use $\diag{\cdot}$ to denote the diagonal matrix with diagonal entries taken from the vector-valued argument.
An optimal solution $\valpha^\star$ to the dual problem in \Cref{eq:svm-dual} corresponds to an optimal primal variable $\vw^\star$ for the problem in \Cref{eq:svm-primal} via the relation $\vw^\star = \sum_{i=1}^n \alpha_i^\star y_i\vx_i$.
The support vectors are precisely the examples $(\vx_i,y_i)$ for which the corresponding $\alpha_i^\star$ is positive, a consequence of complementary slackness.

It will be notationally convenient to change the optimization variable from $\valpha$ to $\vbeta \in \bbR^n$ with $\beta_i = y_i\alpha_i$ for all $i=1,\dotsc,n$.
In terms of $\vbeta$, the SVM dual problem from \Cref{eq:svm-dual} becomes
\begin{equation}
  \begin{aligned}
    \max_{\vbeta \in \bbR^n} \quad & \vy^\T \vbeta - \frac12 \vbeta^\T \vK \vbeta \\
    \text{subj.\ to} \quad & y_i\beta_i \geq 0 \quad \text{for all $i=1,\dotsc,n$} .
  \end{aligned}
  \label{eq:svm-dual2}
\end{equation}
An optimal solution $\vbeta^\star$ to this problem corresponds to an optimal primal variable $\vw^\star$ via the relation $\vw^\star = \sum_{i=1}^n \beta_i^\star \vx_i$, and the support vectors are precisely the examples $(\vx_i,y_i)$ for which $\beta_i^\star$ is non-zero.

Note that if it were not for the $n$ constraints, the solutions to optimization problem would be characterized by the linear equation $\vK\vbeta = \vy$.
We refer to the version of the optimization problem in \Cref{eq:svm-dual2} without the $n$ constraints as the \emph{ridgeless regression problem}.
Solutions to this problem have been extensively studied in recent years~\citep[e.g.,][]{liang2020just,bartlett2020benign,muthukumar2020harmless,belkin2019two,hastie2019surprises,mahdaviyeh2019risk}.
If a vector $\vbeta \in \bbR^n$ satisfies both $\vK\vbeta = \vy$ as well as the $n$ constraints $y_i\beta_i \geq 0$ for all $i=1,\dotsc,n$, then $\vbeta$ is necessarily an optimal solution to the SVM dual problem from \Cref{eq:svm-dual2}.

\subsection{Data model}
\label{sec:data}

We analyze the SVM under the following probabilistic model of the training examples.

\paragraph{Feature model.}
The $\vx_1,\dotsc,\vx_n$ are random vectors in $\bbR^d$ satisfying
\begin{align}
  \vx_i & := \diag{\vlambda}^{1/2} \vz_i , \quad \text{for all $i=1,\dotsc,n$} ,
\end{align}
The positive vector $\vlambda \in \bbR_{++}^d$ parameterizes the model.
The random vectors, collected in the $n \times d$ random matrix $\vZ := [ \vz_1 | \dotsb | \vz_n ]^\T = (z_{i,j})_{1 \leq i \leq n; 1 \leq j \leq d}$, satisfy one of the following distributional assumptions.
\begin{enumerate}
  \item \emph{Independent features:} $\vZ$ has independent entries such that each $z_{i,j}$ is mean-zero, unit variance, and sub-Gaussian with parameter $v>0$ (i.e., $\bbE(z_{i,j}) = 0$, $\bbE(z_{i,j}^2) = 1$, and $\bbE(e^{tz_{i,j}}) \leq e^{vt^2/2}$ for all $t \in \bbR$).

  \item \emph{Haar features:} $\vZ$ is taken to be the first $n$ rows of a uniformly random $d \times d$ orthogonal matrix (with the Haar measure), and then scaled by $\sqrt{d}$.
  The scaling is immaterial to our results, but it makes the analysis comparable to that for the independent features case.
\end{enumerate}

\paragraph{Label model.}
Conditional on $\vx_1,\dotsc,\vx_n$, the $y_1,\dotsc,y_n$ are independent $\{-1,+1\}$-valued random variables such that the conditional distribution of $y_i$ depends only on $\vx_i$ for each $i=1,\dotsc,n$.
Formally:
\begin{equation}
  y_i \independent (\vx_1,y_1,\dotsc,\vx_{i-1},y_{i-1},\vx_{i+1},y_{i+1},\dotsc,\vx_n,y_n) \mid \vx_i .
\end{equation}

\paragraph{Remarks.}
All of our results will assume $d \geq n$.
The non-singularity of the kernel matrix $\vK = \vZ \diag{\vlambda} \vZ^\T$ will be important for our analysis.
In the case of Haar features, setting $d \geq n$ ensures that the matrix $\vZ$ always has rank $n$, and hence the kernel matrix $\vK = \vZ \diag{\vlambda} \vZ^\T$ is always non-singular.
In the case of independent features, if the distributions of the $z_{i,j}$ are continuous, then $\vZ$ has rank $n$ almost surely, and hence again $\vK$ is non-singular almost surely.
Our results only require the $z_{i,j}$ to be sub-Gaussian and need not have continuous distributions.
For instance, if the $z_{i,j}$ are Rademacher (uniform on $\{-1,+1\}$), then there is a non-zero probability that $\vZ$ is rank-deficient---however, we will see that this probability is negligible.

Our label model is very general and allows for a variety of settings, including the following.
\begin{enumerate}
  \item \emph{Generalized linear models (GLMs):} $\Pr(y_i = 1 \mid \vx_i) = g(\vx_i^\T\vw)$ for some $\vw \in \bbR^d$ and some function $g \colon \bbR \to [0,1]$.
    Examples include \emph{logistic regression}, where $g(t) = 1/(1+e^{-t})$; \emph{probit regression}, where $g(t) = \Phi(t)$ and $\Phi$ is the cumulative distribution function of the standard Gaussian distribution; and \emph{one-bit compressive sensing}~\citep{boufounos2008one}, where $g(t) = \ind{t>0}$.

  \item \emph{Multi-index models:} $\Pr(y_i = 1 \mid \vx_i) = h(\vW\vx_i)$ for some $k \in \bbN$, $\vW \in \bbR^{k \times d}$, and $h \colon \bbR^k \to [0,1]$.
   The case $k=1$ corresponds to GLMs.
   Examples with $k\geq 2$ include the intersections of half-spaces models and certain neural networks~\citep{baum1990polynomial,klivans2004learning,klivans2008learning}.

  \item \emph{Fixed labels:} $y_i \in \{-1,+1\}$ are fixed (non-random) values.
  This can be regarded as a null model where the feature vectors have no statistical relationship to the labels.
  This null model was, e.g., considered by~\citet{cover1965geometrical}.

\end{enumerate}

Our results in \Cref{thm:independent} and \Cref{thm:haar} consider, respectively, the independent features and Haar features, but both allowing for general label models.
Our weak converse result in \Cref{thm:converse} is established in the special case where the $z_{i,j}$ are iid standard Gaussian random variables (a special case of independent features), and where the labels are fixed.

\subsection{Additional notation}

Let $[n] := \{1,\dotsc,n\}$ for any natural number $n$.
Let $\bbR_{++} := \{ x \in \bbR \colon x > 0 \}$ denote the positive real numbers.
For a vector $\vv \in \bbR^n$, we let $\vv_{\lo i} \in \bbR^{n-1}$ denote the vector obtained from $\vv$ by omitting the \ith coordinate.
For a matrix $\vM \in \bbR^{n \times d}$, we let $\vM_{\lo i} \in \bbR^{(n-1) \times d}$ denote the matrix obtained from $\vM$ by omitting the \ith row.
Sometimes, for a square matrix $\vM \in \bbR^{n \times n}$, we will also use $\vM_{\lo i} \in \bbR^{(n-1) \times (n-1)}$ to denote the matrix obtained from $\vM$ by removing the \ith row and column.
We let $\ve_i$ denote the \ith coordinate vector in $\bbR^n$.
For a vector $\vv \in \bbR^d$, we denote its $p$-norm by $\|\vv\|_p = (\sum_{i=1}^d |v_i|^p)^{1/p}$.
For a matrix $\vM \in \bbR^{d \times d}$, we denote its $2\to2$ operator norm (i.e., largest singular value) by $\|\vM\|_{\op} = \sup_{\vv \in \bbR^d : \|\vv\|_2 \leq 1} \|\vM\vv\|_2$.
Let $S^{d-1} := \{ \vx \in \bbR^d : \|\vx\|_2 = 1 \}$ denote the unit sphere in $\bbR^d$.
If $\vM$ is a symmetric matrix, $\eigmin(\vM)$ denotes the smallest eigenvalue of $\vM$.
Finally, we will use $(C,c,c_1,c_2)$ to denote universal constants that do not depend, explicitly or implicitly, on the dimension $d$, the number of training examples $n$, or properties of the data distribution.

\section{Main results}
\label{sec:results}

Our primary interest is in the probability that every training example is a support vector under the data model from \Cref{sec:data}.
We give sufficient conditions on certain effective dimensions for this probability to tend to one as $n \to \infty$.
We complement these results with a partial weak converse.
Finally, we present a key deterministic result that is used in the proofs of the aforementioned results.
All proofs are given in \Cref{sec:proofs}.

We define the following effective dimensions in terms of the data model parameter $\vlambda$:
\begin{equation}
  d_2 := \frac{\|\vlambda\|_1^2}{\|\vlambda\|_2^2}
  \qquad \text{and} \qquad
  d_\infty := \frac{\|\vlambda\|_1}{\|\vlambda\|_\infty} .
\end{equation}
Observe that $d \geq d_2 \geq d_\infty$, and that if $\lambda_j = 1$ for all $j=1,\dotsc,d$ (i.e., the isotropic setting), then $d = d_2 = d_\infty$.
We note that $d_2$ and $d_\infty$ are, respectively, the same as the effective ranks $r_0(\diag{\vlambda})$ and $R_0(\diag{\vlambda})$ studied by \citet{bartlett2020benign}.
They arise naturally from the tail behavior of certain linear combinations of $\chi^2$-random variables~\citep[see, e.g.,][]{laurent2000adaptive}.

\subsection{Sufficient conditions}

Our first main result provides sufficient conditions on the effective dimensions $d_2$ and $d_\infty$ in the independent features setting so that, with probability tending to one, every training example is a support vector.

\begin{theorem}
  \label{thm:independent}
  There are universal constants $C>0$ and $c>0$ such that the following holds.
  If the training data $(\vx_1,y_1),\dotsc,(\vx_n,y_n)$ follow the model from \Cref{sec:data} with independent features, subgaussian parameter $v>0$, and model parameter $\vlambda \in \bbR_{++}^d$, then the probability that every training example is a support vector is at least
  \begin{equation}
    1 - \exp\del{ - c \cdot \min\cbr{ \frac{d_2}{v^2} ,\, \frac{d_\infty}{v} } + Cn } - \exp\del{ - c \cdot \frac{d_\infty}{vn} + C\log n } .
  \end{equation}
\end{theorem}

Observe that the probability from \Cref{thm:independent} is close to $1$ when
\begin{equation}
  d_2 \gg v^2 n
  \qquad \text{and} \qquad
  d_\infty \gg v n \log n .
\end{equation}
We can compare this condition to that from the prior work of \citet{muthukumar2020classification} in our setting with independent Gaussian features ($v=1$).
In the anisotropic setting (i.e., general $\vlambda$), the prior result's condition for every training example to be a support vector with high probability is $d_2 \gg n^2 \log n$ and $d_\infty \gg n^{3/2} \log n$.
In the isotropic setting (i.e., all $\lambda_j=1$), assuming the labels are fixed (i.e., non-random), the prior result's condition is $d \gg n \log n$.
\Cref{thm:independent} is an improvement in the anisotropic case, and it matches this prior result in the isotropic case.\footnote{%
  We remark that the result of \citet{muthukumar2020classification} for the anisotropic case, in fact, holds for all (fixed) label vectors $\vy \in \{-1,+1\}^n$ simultaneously.
  However, their proof does not readily give a tighter condition when only a single (random) label vector is considered.
  Our proof technique side-steps this issue by showing that it is sufficient to consider the scaling of quantities that do not depend on the value of the label vector.}

Our second main result provides an analogue of \Cref{thm:independent} for the case of Haar features (where neither training examples nor features are statistically independent).

\begin{theorem} \label{thm:haar}
  There are universal constants $C>0$ and $c>0$ such that the following holds.
  If the training data $(\vx_1,y_1),\dotsc,(\vx_n,y_n)$ follow the model from \Cref{sec:data} with Haar features and model parameter $\vlambda \in \bbR_{++}^d$, then the probability that every training example is a support vector is at least
  \begin{equation}
    1 - \exp\del{ -c \cdot d_\infty + Cn } - \exp\del{ -c \cdot \frac{d-n+1}{d} \cdot \frac{d_\infty}{n} + C \log n } .
  \end{equation}
\end{theorem}

\subsection{Weak converse}

Our final main result gives a weak converse to \Cref{thm:independent} in the case where the features are iid standard Gaussians and the labels are fixed.

\begin{theorem} \label{thm:converse}
  Let the training data $(\vx_1,y_1),\dotsc,(\vx_n,y_n)$ follow the model from \Cref{sec:data} with $\vlambda = (1,\dotsc,1)$, $\vz_1,\dotsc,\vz_n$ being iid standard Gaussian random vectors in $\bbR^d$, and $y_1,\dotsc,y_n \in \{\pm1\}$ being arbitrary but fixed (i.e., non-random) values.
  For any $d \geq n$, the probability that at least one training example is not a support vector is at least
  \begin{equation}
    \Phi\del{ -\sqrt{\frac{d-n+4+2\sqrt{d-n+2}}{n-1}} } \cdot \del{ 1 - \frac1e } ,
  \end{equation}
  where $\Phi$ is the cumulative distribution function of the standard Gaussian distribution.
\end{theorem}
Observe that the probability bound from \Cref{thm:converse} is at least a positive constant (independent of $d$ and $n$) whenever the dimension $d$ (regarded as a function of $n$) is $O(n)$.
This means that the dimension $d$ must be super-linear in $n$ in order for the ``success'' probability of \Cref{thm:independent} to tend to one with $n$.

\Cref{thm:converse} applies to the case where the features vectors are isotropic.
In \Cref{sec:anisotropic}, we give a version of the result that applies to certain anisotropic settings, again in the case of independent Gaussian features and fixed labels.
The theorem puts restrictions on the tail behavior of $\vlambda$.
These restrictions are related to the effective ranks studied by \citet{bartlett2020benign}.
The proof is similar to that of \Cref{thm:converse}, but also relies on a technical result from \citep{bartlett2020benign}.

Except when the ``success'' probability is required to be $\geq 1-1/n^c$ for constant $c>0$, there is a $\log(n)$ gap between the sufficient condition from \Cref{thm:independent} and the necessary condition from \Cref{thm:converse}.
As mentioned earlier, subsequent work of \citet{ardeshir2021support} closed this gap and showed that even for constant success probability, $d = \Omega(n \log n)$ is necessary.

\subsection{Deterministic equivalences}

The crux of all of the above results lies in the following key lemma, which characterizes equivalent conditions for every training example to be a support vector.

\begin{lemma}
  \label{lem:equivalent}
  Suppose $\vZ := [ \vz_1 | \dotsb | \vz_n ]^\T \in \bbR^{n \times d}$ and $\vlambda \in \bbR_{++}^d$ are such that $\vZ \diag{\vlambda} \vZ^\T$ and $\vZ_{\lo i} \diag{\vlambda} \vZ_{\lo i}^\T$ for all $i = 1,\dotsc,n$ are non-singular.
  Let the training data $(\vx_1,y_1),\dotsc,(\vx_n,y_n) \in \bbR^d \times \{-1,+1\}$ satisfy $\vx_i = \diag{\vlambda}^{1/2} \vz_i$ for each $i = 1,\dotsc,n$.
  Then the following are equivalent:
  \begin{enumerate}
    \item Every training example is a support vector.
    \item The vector $\vbeta := \vK^{-1} \vy$ satisfies $y_i\beta_i > 0$ for all $i=1,\dotsc,n$.
    \item $y_i \vy_{\lo i}^\T \del{ \vZ_{\lo i} \diag{\vlambda} \vZ_{\lo i}^\T }^{-1} \vZ_{\lo i} \diag{\vlambda} \vz_i < 1$ for all $i=1,\dotsc,n$.
  \end{enumerate}
\end{lemma}

The above lemma is a deterministic result---it does not reference a particular statistical model for the data---and hence the equivalences are given under non-singularity conditions.
We note that the non-singularity conditions are readily satisfied under the data model from \Cref{sec:data} (with high probability, in the case of independent features, or deterministically, in the case of Haar features).

The equivalences of the first two items in this lemma connect the solutions to the SVM optimization problem and the ridgeless regression problem more tightly than was done in the prior work of \citet{muthukumar2020classification}, who only proved one direction of the equivalence between the first two items.
The proofs of our main results critically use the third item in the above equivalence.

\subsection{Implications for generalization}
\label{sec:discussion}

In \Cref{thm:independent} and \Cref{thm:haar}, we identified high-dimensional regimes in which the SVM solution exactly corresponds to the least norm (linear) interpolation of training data with high probability.
We observe in \Cref{fig:fourier} that certain deterministic featurizations (which bear some resemblance to the Haar features of \Cref{thm:haar}, and have been independently analyzed in the interpolating regime for regression problems~\citep{belkin2019two,muthukumar2020harmless}) also empirically exhibit similar support vector proliferation when the effective overparameterization is sufficiently high.

\begin{figure}
  \centering
  \begin{tabular}{cc}
    \includegraphics[width=0.48\textwidth]{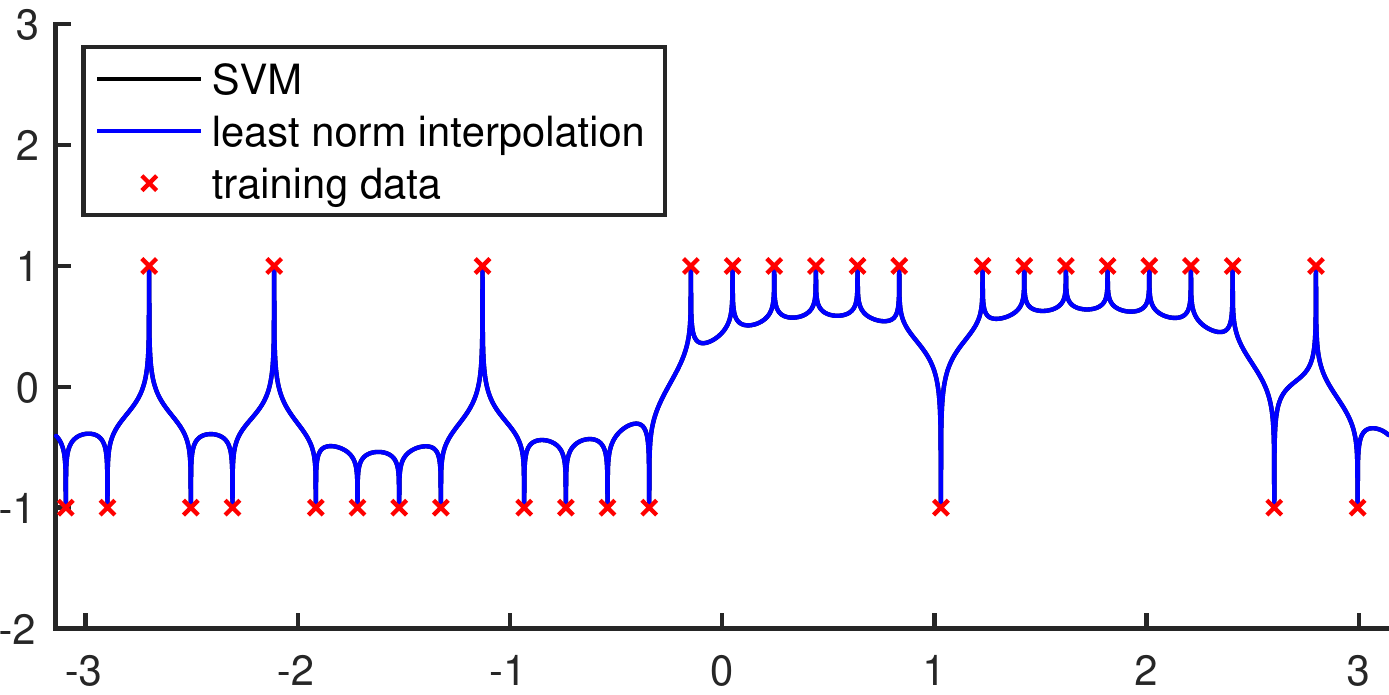}
    &
    \includegraphics[width=0.48\textwidth]{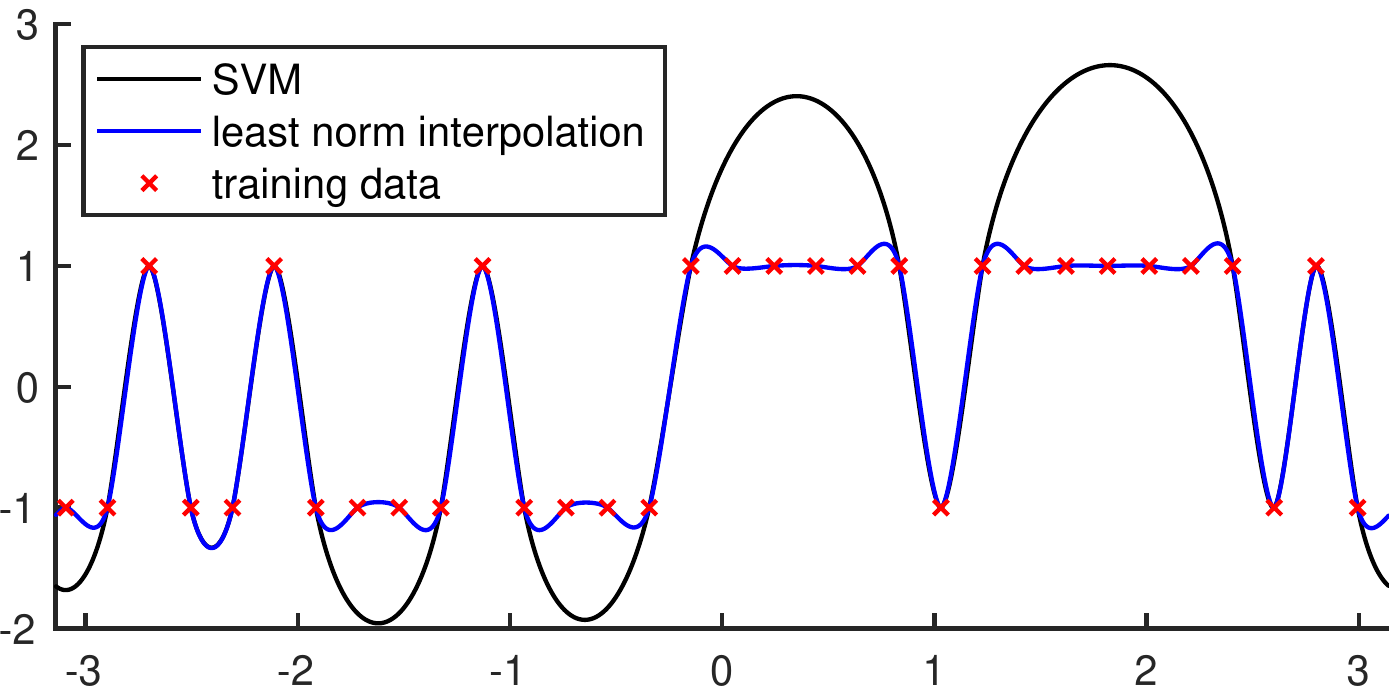}
    \\
    (a) $\eta_i = 1/i$
    &
    (b) $\eta_i = 1/i^3$
  \end{tabular}
  \caption{%
    Plots of linear functions on top of trigonometric features of a scalar input variable that parameterizes the horizontal axis.
    (These plots originally appeared in \citep{muthukumar2020classification}.)
    The two linear functions are those given by the solution to the SVM optimization problem and the ridgeless regression problem (i.e., the least norm interpolation), based on $32$ training data shown as \textcolor{red}{$\times$}'s in the plot.
    The features are obtained via the mapping $t \mapsto (1, \sqrt{\eta_1} \cos(1 \cdot t), \sqrt{\eta_1} \sin(1 \cdot t), \dotsc, \sqrt{\eta_k} \cos(k \cdot t), \sqrt{\eta_k} \sin(k \cdot t)) \in \bbR^{2k+1}$ where $k=2^{14}$.
    In (a), the SVM and least norm interpolation coincide exactly (so all $32$ examples are support vectors); in (b), the functions are noticeably distinct (and only $18$ out of $32$ examples are support vectors).
    In each case, we computed analogues of $d_2$ and $d_\infty$ based on the eigenvalues of the Gram matrix.
    In (a), they are $108.386$ and $21.5626$; in (b), they are $3.21378$ and $2.20198$.%
  }
  \label{fig:fourier}
\end{figure}

The regimes considered in our results go beyond the common high-dimensional asymptotic where $d$ and $n$ grow proportionally to each other (i.e., $n/d \to \delta$ as $n,d \to \infty$).
One may wonder, then, whether these regimes are too high dimensional for the SVM to generalize well.
As mentioned in \Cref{sec:intro}, the classical generalization error bounds for the SVM are based on the number of support vectors or the worst-case margin achieved on the training examples.
Recall that these upper bounds are, respectively, roughly of the form\footnote{Some bounds are given as the square-roots of the expressions we show, but whether or not the square-root is used will not make a difference in our case. We also omit constants (which are typically larger than $1$), polylogarithmic factors in $n$, and terms related to the confidence level for the bound.}
\begin{equation}
  \frac{\text{\# support vectors}}{n}
  \qquad\text{and}\qquad
  \frac{\|\vw^\star\|_2^2}{n} \cdot \frac{\bbE[\tr{\vK}]}{n}
  .
\end{equation}
Here, $\vw^\star$ is the solution to the SVM primal problem in \Cref{eq:svm-primal}.
Unfortunately, these bounds are not informative for the high-dimensional regimes in which all training points become support vectors.
As soon as $d_2$ and $d_\infty$, respectively, grow beyond $n$ and $n \log n$, then both bounds above become trivial with probability tending to one.
This is immediately apparent for the first bound, as a consequence of \Cref{thm:independent}.
For the second bound, an inspection of the proof of \Cref{thm:independent} shows that in an event where every training example is a support vector (with the same probability as given in \Cref{thm:independent}), we have
\begin{align}
  \|\vw^\star\|_2^2 & = \vy^\T \vK^{-1} \vy \geq \frac{n}{\|\vK\|_{\op}} \geq \frac{n}{2\|\vlambda\|_1} .
\end{align}
Since $\bbE[ \tr[0]{\vK} ] / n = \|\vlambda\|_1$, the second bound is at least $1/2$ in this event.
We also remark that even more sophisticated generalization bounds using the distribution of the margin on training examples~\citep[e.g.,][]{gao2013doubt} do not help in this high-dimensional regime.
This is because when all training examples become support vectors, the normalized margin of every training point becomes exactly the worst-case margin, which is $1/\|\vw^\star\|_2$.

However, recent analyses show that the SVM can generalize well even when all training points become support vectors.
In particular, the recent work of~\citet{muthukumar2020classification} provided positive implications for the SVM by analyzing the classification test error of the least norm interpolation.
In particular, they considered a special anisotropic Gaussian ensemble inspired by spiked covariance models, parameterized by positive constants $p > 1$ and $0 < (q,r) < 1$; here, $d = n^p$ and $(q,r)$ parameterize the eigenvalues of the feature covariance matrix and the sparsity of the unknown signal respectively.
See~\cite[Section~3.4]{muthukumar2020classification} for further details.
It suffices for our purposes to note that the main result of~\citet[Theorem 2]{muthukumar2020classification} showed that the following \emph{rate region} of $(p,q,r)$ is necessary and sufficient for the least norm interpolation of training data to generalize well, in the sense that the classification test error goes to $0$ as $n \to \infty$: 
\begin{align}\label{eq:rateregion}
  0 & \leq q < 1 - r + \frac{p-1}{2} .
\end{align}
It is easy to verify that \Cref{thm:independent} directly implies good generalization of the SVM for this entire rate region.
First, for $q \geq 1-r$, it holds that
\begin{align}
  d_2 & \asymp n^{2p - \max\{2p - 2q - r, p\}} \\
  d_{\infty} & \asymp n^{q + r} ,
\end{align}
and since we have assumed $p > 1$, the conditions of \Cref{thm:independent}, i.e., $d_2 \gg n$, $d_{\infty} \gg n \log n$, would hold if and only if $q > 1-r$.
On the other hand, the usual margin-based bounds would show good generalization of the SVM \emph{if} $0 \leq q < (1-r)$.
Putting these together, the SVM generalizes well for the entire rate region in Equation~\eqref{eq:rateregion}.

Further, the improvement of this implication over the partial implications for the SVM that were provided in~\citet{muthukumar2020classification} is clear.
In particular,~\cite[Corollary 1]{muthukumar2020classification} required $p > 2$, i.e. $d \gg n^2$, and showed that the SVM will then generalize well if $(3/2 - r) < q < (1-r) + (p-1)/2$.
Thus, the rate region implied by this work was
\begin{align}\label{eq:rateregionsuboptimal}
    \left\{0 \leq q < (1-r)\right\} \cup \left\{\left(\frac{3}{2} - r\right) < q < (1-r) + \frac{(p-1)}{2}\right\} ,
\end{align}
which has a non-trivial gap compared to \Cref{eq:rateregion}.
In summary, our results imply an expansion over the rate region predicted by classical generalization bounds based on either the number of support vectors or the margin.

\section{Proofs}
\label{sec:proofs}

This section gives the proofs of the main results, as well as the proof of the main technical lemma.

Throughout, we use the shorthand notations $\vLambda := \diag{\vlambda}$ and $\vK_{\lo i} := \vZ_{\lo i} \vLambda \vZ_{\lo i}^\T$ for each $i = 1,\dotsc,n$.
Note that $\vK_{\lo i}$ is the same as $\vK = \vZ \diag{\vlambda} \vZ^\T$ except omitting both the \ith row \emph{and} the \ith column (whereas $\vZ_{\lo i}$ only omits the \ith row of $\vZ$).

\subsection{Proof of \Cref{lem:equivalent}}

Recall that we assume $\vK$ and $\vK_{\lo i}$ for all $i = 1,\dotsc,n$ are non-singular.
We first show that all training examples are support vectors if and only if the candidate solution $\vbeta = \vK^{-1} \vy$ satisfies
\begin{align}
  \label{eq:complementaryslacknesscondition}
  y_i \beta_i > 0 \quad \text{for all $i = 1,\dotsc,n$} .
\end{align}
\begin{itemize}
  \item 
    ($\impliedby$)
    Assume $y_i\beta_i > 0$ for all $i \in [n]$.
    Recall that $\vbeta = \vK^{-1} \vy$ is the unique optimal solution to the ridgeless regression problem (i.e., the problem in \Cref{eq:svm-dual2} without the $n$ constraints).
    Since \Cref{eq:complementaryslacknesscondition} holds, then $\vbeta$ is dual-feasible as well, and so it is the unique optimal solution to the dual program, i.e., $\vbeta^\star = \vbeta$.
    Moreover, $y_i \beta^\star_i > 0 \implies \beta^\star_i \neq 0$ for all $i \in [n]$, and so every training example is a support vector.

  \item 
    ($\implies$)
    Assume every training example is a support vector, i.e., $\beta_i^\star \neq 0$ for all $i \in [n]$ (so, in particular, $y_i\beta_i^\star > 0$ for all $i \in [n]$).
    We shall write the solution $\vw^\star$ to the primal problem from \Cref{eq:svm-primal} as a linear combination of $\vx_1,\dotsc,\vx_n$ in two ways.
    The first way is in terms of the dual solution $\vbeta^\star$, i.e., $\vw^\star = \sum_{i=1}^n \beta_i^\star \vx_i$, which follows by strong duality.
    The second way comes via complementary slackness, which implies that $\vw^\star$ satisfies every constraint in \Cref{eq:svm-primal} with equality.
    In other words, $\vw^\star$ solves
    \begin{equation}
      \begin{aligned}
        \min_{\vw \in \bbR^d} \quad & \frac12 \|\vw\|_2^2 \\
        \text{subj.\ to} \quad & \vx_i^\T\vw = y_i \quad \text{for all $i=1,\dotsc,n$} .
      \end{aligned}
      \label{eq:svm-primal-equality}
    \end{equation}
    Since $\vK$ is non-singular by assumption, the solution is unique and is given by $\vX^\T \vK^{-1} \vy = \vX^\T \vbeta = \sum_{i=1}^n \beta_i \vx_i$, where $\vX = [ \vx_1 | \dotsb | \vx_n ]^\T$.
    So we have $\vw^\star = \sum_{i=1}^n \beta_i^\star \vx_i = \sum_{i=1}^n \beta_i \vx_i$.
    The non-singularity of $\vK$ also implies that $\vx_1,\dotsc,\vx_n$ are linearly independent, so we must have $\beta_i = \beta_i^\star \neq 0$ for all $i \in [n]$, and thus \Cref{eq:complementaryslacknesscondition} holds.

\end{itemize}

So we have shown that all training examples are support vectors if and only if \Cref{eq:complementaryslacknesscondition} holds.
It therefore suffices to show that, for each $i=1,\dotsc,n$,
\begin{equation}
  y_i \beta_i > 0 \quad \iff \quad y_i \vy_{\lo i}^\T \vK_{\lo i}^{-1} \vZ_{\lo i} \vLambda \vz_i < 1
   .
\end{equation}
By symmetry, we only need to show this implication for $i=1$.

Observe that $y_1 \beta_1 = y_1\ve_1^\T\vK^{-1} \vy = \ve_1^\T\vK^{-1} (y_1\vy)$ is the inner product between the first row of $\vK^{-1}$ and $y_1\vy$.
Therefore, by Cramer's rule, we have
\begin{align}
  y_1\beta_1
  & = y_1\ve_1^\T\vK^{-1} \vy
  = \frac{\det(\tilde\vK)}{\det(\vK)}
\end{align}
where $\tilde\vK$ is the matrix obtained from $\vK$ by replacing the first row with $y_1\vy^\T$.
Since $\vK$ is assumed to be invertible, $\vK$ is positive definite, and so $\det(\vK) > 0$.
Hence, we have $y_1\beta_1 > 0$ \emph{iff} $\det(\tilde\vK) > 0$.

Let us write $\tilde\vK$ as
\begin{align}
  \tilde\vK
  & =
  \begin{bmatrix}
    1 & y_1\vy_{\lo 1}^\T \\
    \va & \vK_{\lo 1}
  \end{bmatrix} ,
\end{align}
where $\va := \vZ_{\lo 1} \vLambda \vz_1$ and recall that $\vK_{\lo 1}$ denotes the $(n-1) \times (n-1)$ matrix obtained by removing the first row and column from $\vK$.
Note that $\vK_{\lo 1}$ is invertible by assumption and hence positive definite.
Also, define
\begin{align}
  \vQ & :=
  \begin{bmatrix}
    1 & -y_1 \vy_{\lo 1}^\T \\
    \v0 & \vI_{n-1}
  \end{bmatrix} ,
\end{align}
where $\vI_{n-1}$ is the $(n-1)\times(n-1)$ identity matrix.
Every diagonal entry of $\vQ$ is equal to $1$, so $\det(\vQ) = 1$.
Hence
\begin{align}
  \det(\tilde\vK)
  & = \det(\tilde\vK) \det(\vQ) \\
  & = \det(\tilde\vK\vQ) \\
  & = \det\del{
    \begin{bmatrix}
      1 & \v0^\T \\
      \va & \vK_{\lo 1} - y_1 \va \vy_{\lo 1}^\T
    \end{bmatrix}
  } \\
  & = \det(\vK_{\lo 1} - \va\vb^\T)
\end{align}
where $\vb := y_1 \vy_{\lo 1}$.
Therefore, $\det(\tilde\vK) > 0$ \emph{iff} $\det(\vK_{\lo 1} - \va\vb^\T) > 0$.

By the matrix determinant lemma,
\begin{align}
  \det(\vK_{\lo 1} - \va\vb^\T)
  & = \det(\vK_{\lo 1}) (1 - \vb^\T \vK_{\lo 1}^{-1} \va) .
\end{align}
Since $\vK_{\lo 1}$ is positive definite, we have $\det(\vK_{\lo 1}) > 0$.
Hence, $\det(\vK_{\lo 1} - \va\vb^\T) > 0$ \emph{iff} $\vb^\T \vK_{\lo 1}^{-1} \va < 1$.

Connecting all of the equivalences and plugging-in for $\va$, $\vb$, and $\vK_{\lo 1}$, we have shown that
\begin{equation}
  y_1\beta_1 > 0
  \quad \iff \quad
  y_1\vy_{\lo 1}^\T (\vZ_{\lo 1} \vLambda \vZ_{\lo 1}^\T)^{-1} \vZ_{\lo 1}\vLambda\vz_1 < 1 ,
\end{equation}
as required.
This completes the proof of the lemma.
\qed

\subsection{Proof of \Cref{thm:independent}}

We fix $t = t(n,\vlambda)>0$ to a positive value depending on $\vlambda$ and $n$ that will be determined later.
We define the following events:
\begin{enumerate}
  \item For $i \in [n]$, $\cB_i$ is the event that $\vK_{\lo i}$ is non-singular and
    \begin{align}
      y_i \vy_{\lo i}^\T \vK_{\lo i}^{-1} \vZ_{\lo i} \vLambda \vz_i & \geq 1 .
    \end{align}

  \item For $i \in [n]$, $\cS_i$ is the event that $\vK_{\lo i}$ is singular.

  \item $\cS$ is the event that $\vK$ is singular.

  \item $\cB := \cS \cup \bigcup_{i=1}^n (\cB_i \cup \cS_i)$.

\end{enumerate}

Additionally, we define the event $\cE_i(t)$, for every $i \in [n]$ and a given $t > 0$, that $\vK_{\lo i}$ is non-singular and
    \begin{align}
      \norm{ \vLambda \vZ_{\lo i}^\T \vK_{\lo i}^{-1} \vy_{\lo i} }_2^2 & \geq \frac1t .
    \end{align}

Note that if the event $\cB$ does not occur, then $\vZ \vLambda \vZ^\T$ is non-singular, each $\vK_{\lo i}$ is non-singular, and
\begin{align}
  y_i \vy_{\lo i}^\T \vK_{\lo i}^{-1} \vZ_{\lo i} \vLambda \vz_i & < 1 , \quad \text{for all $i=1,\dotsc,n$} .
\end{align}
Hence, by \Cref{lem:equivalent}, if $\cB$ does not occur, then every training example is a support vector.

So, it suffices to upper-bound the probability of the event $\cB$.
We bound $\Pr(\cB)$ as follows:
\begin{align}
  \Pr(\cB)
  & \leq \Pr(\cS) + \sum_{i=1}^n \Pr(\cB_i \cup \cS_i) \\
  & = \Pr(\cS) + \sum_{i=1}^n \del{ \Pr((\cB_i \cap \cS_i^\compl \cap \cE_i(t)^\compl) \cup (S_i \cap \cE_i(t)^\compl)) + \Pr((\cB_i \cup \cS_i) \cap \cE_i(t)) } \\
  & \leq \Pr(\cS) + \sum_{i=1}^n \del{ \Pr(\cB_i \mid \cS_i^\compl \cap \cE_i(t)^\compl) \Pr(\cS_i^\compl \cap \cE_i(t)^\compl) + \Pr(S_i \cap \cE_i(t)^\compl) + \Pr((\cB_i \cup \cS_i) \cap \cE_i(t)) } \\
  & \leq \Pr(\cS) + \sum_{i=1}^n \del{ \Pr(\cB_i \mid \cS_i^\compl \cap \cE_i(t)^\compl) + \Pr(S_i) + \Pr(\cE_i(t)) }
  .
  \label{eq:prob-bad-events}
\end{align}
Above, the first two inequalities follow from the union bound, and the rest uses the law of total probability.

We first upper bound the probability of the singularity events in the following lemma.
\begin{lemma}
  \label{lem:prob-singular}
  We have
  \begin{align}
    \max\{ \Pr(\cS) , \Pr(\cS_1) , \dotsc, \Pr(\cS_n) \}
    & \leq 2 \cdot 9^n \cdot \exp\del{ -c \cdot \min\cbr{ \frac{d_2}{v^2} ,\, \frac{d_\infty}{v} } }
  \end{align}
  where $c>0$ is the universal constant in the statement of \Cref{lem:randmatrix}.
\end{lemma}
\begin{proof}
  It suffices to bound $\Pr(\cS)$, since each $\vK_{\lo i}$ is a principal submatrix of $\vK$, and hence $\eigmin(\vK_{\lo i}) \geq \eigmin(\vK)$ for all $i \in [n]$.
  Observe that
  \begin{align}
    \vZ \vLambda \vZ^\T
    & = \sum_{j=1}^d \lambda_j \vv_j \vv_j^\T
  \end{align}
  where $\vv_j$ is the \jth column of $\vZ$.
  Recall that the columns of $\vZ$ are independent, and so these vectors satisfy the conditions of \Cref{lem:randmatrix}.
  Moreover, since $\vZ \vLambda \vZ^\T$ is positive semi-definite, its singularity would require
  \begin{align}
    \norm{\vZ \vLambda \vZ^\T - \|\vlambda\|_1 \vI}_2
    & \geq \|\vlambda\|_1 .
  \end{align}
  The probability of this latter event can be bounded by \Cref{lem:randmatrix} with $\tau=\|\vlambda\|_1$, thereby giving the claimed bound on $\Pr(\cS)$.
  This completes the proof of the lemma.
\end{proof}

The next lemma upper bounds the probability of the event $\cB_i$ conditioned on the non-singularity event $\cS_i$ and the complement of the event $\cE_i(t)$.
\begin{lemma}
  \label{lem:prob-B_i}
  For any $t>0$,
  \begin{align}
    \Pr(\cB_i \mid \cS_i^\compl \cap \cE_i(t)^\compl) & \leq 2\exp\del{ -\frac{t}{2v} } .
  \end{align}
\end{lemma}
\begin{proof}
  Let $\cB_i'$ be the event that $\vK_{\lo i}$ is non-singular and
  \begin{align}
    \abs[1]{ \vy_{\lo i}^\T \vK_{\lo i}^{-1} \vZ_{\lo i} \vLambda \vz_i }
    & = \max\cbr[1]{ -\vy_{\lo i}^\T \vK_{\lo i}^{-1} \vZ_{\lo i} \vLambda \vz_i ,\, \vy_{\lo i}^\T \vK_{\lo i}^{-1} \vZ_{\lo i} \vLambda \vz_i }
    \geq 1 .
  \end{align}
  Since $|y_i| = 1$, it follows that $\cB_i \subseteq \cB_i'$, so
  \begin{align}
    \Pr(\cB_i \mid \cS_i^\compl \cap \cE_i(t)^\compl)
    & \leq \Pr(\cB_i' \mid \cS_i^\compl \cap \cE_i(t)^\compl) .
  \end{align}
  Conditional on the event $\cS_i^\compl \cap \cE_i(t)^\compl$, we have that $\vK_{\lo i}$ is non-singular and $\|\vLambda \vZ_{\lo i}^\T \vK_{\lo i}^{-1} \vy_{\lo i}\|_2^2 \leq 1/t$.
  Since $\vz_i$ is independent of $\{ (\vz_j,y_j) : j \neq i \}$, it follows that
  \begin{align}
    \vy_{\lo i}^\T \vK_{\lo i}^{-1} \vZ_{\lo i} \vLambda \vz_i
    = (\vLambda \vZ_{\lo i}^\T \vK_{\lo i}^{-1} \vy_{\lo i})^\T \vz_i
  \end{align}
  is (conditionally) sub-Gaussian with parameter at most $v \cdot \|\vLambda \vZ_{\lo i}^\T \vK_{\lo i}^{-1} \vy_{\lo i}\|_2^2 \leq v/t$.
  Then, the standard sub-Gaussian tail bound gives us
  \begin{align}
    \Pr\del{ \cB_i \mid \cS_i^\compl \cap \cE_i(t)^\compl }
    & \leq \Pr\del{ \cB_i' \mid \cS_i^\compl \cap \cE_i(t)^\compl }
    \leq 2\exp\del{-\frac{t}{2v}} .
  \end{align}
This completes the proof of the lemma.
\end{proof}

Finally, the following lemma upper bounds the probability of the event $\cE_i(t)$ for $t := d_{\infty}/2n$.
\begin{lemma}
  \label{lem:prob-E_i}
  \begin{align}
    \Pr(\cE_i(d_\infty/(2n)))
    & \leq 2 \cdot 9^{n-1} \cdot \exp\del{ -c \cdot \min\cbr{ \frac{d_2}{4v^2} ,\, \frac{d_\infty}{v} } }
  \end{align}
  where $c>0$ is the universal constant from \Cref{lem:randmatrix}.
\end{lemma}
\begin{proof}
  Let $\cE_i'(t)$ be the event that
  \begin{align}
    \eigmin(\vK_{\lo i})
    & \leq n\|\vlambda\|_\infty t .
  \end{align}
  Under $\cS_i^\compl$, the matrix $\vK_{\lo i}$ is non-singular.
  We get
  \begin{align}
    \|\vLambda \vZ_{\lo i}^\T \vK_{\lo i}^{-1} \vy_{\lo i}\|_2^2
    & \leq \|\vLambda^{1/2}\|_{\op}^2
    \|\vLambda^{1/2} \vZ_{\lo i}^\T \vK_{\lo i}^{-1} \vy_{\lo i}\|_2^2
    \\
    & = \|\vlambda\|_\infty
    \vy_{\lo i}^\T \vK_{\lo i}^{-1} \vZ_{\lo i} \vLambda \vZ_{\lo i}^\T \vK_{\lo i}^{-1} \vy_{\lo i}
    \\
    & \leq
    n \|\vlambda\|_\infty
    \sup_{\vu \in \bbR^{n-1} : \|\vv\|_2 = 1}
    \vu^\T \vK_{\lo i}^{-1} \vu
    \\
    & = \frac{n\|\vlambda\|_\infty}{\eigmin(\vK_{\lo i})}
    .
  \end{align}
  It follows that $\cE_i(t) \subseteq \cE_i'(t)$.
  Observe that for $t := d_\infty/(2n)$, the event $\cE_i'(t)$ is that where
  \begin{align}
    \eigmin(\vK_{\lo i}) & \leq \frac12 \|\vlambda\|_1 .
  \end{align}
  Therefore (as in the proof of \Cref{lem:prob-singular}), \Cref{lem:randmatrix} with $\tau=\|\vlambda\|_1/2$ implies that
  \begin{align}
    \Pr(\cE_i'(d_\infty/(2n)))
    & = \Pr\del{ \eigmin(\vK_{\lo i}) \leq \frac12\|\vlambda\|_1 } \\
    & \leq 2 \cdot 9^{n-1} \cdot \exp\del{ -c \cdot \min\cbr{ \frac{d_2}{4v^2} ,\, \frac{d_\infty}{v} } } .
  \end{align}
This completes the proof of the lemma.
\end{proof}

Plugging the probability bounds from \Cref{lem:prob-singular}, \Cref{lem:prob-B_i} and \Cref{lem:prob-E_i} (with $t=d_\infty/(2n)$) into \Cref{eq:prob-bad-events} completes the proof of \Cref{thm:independent}.
\qed

\subsection{Proof of \Cref{thm:haar}}

The proof follows a similar sequence of steps to that of~\Cref{thm:independent} with slight differences in the events that we condition on.
We first observe that $\frac{1}{\sqrt{d}} \vz_i \mid (\vZ_{\lo i}, \vy_{\lo i})$ is a uniformly random unit vector in $S^{d-1}$ restricted to the subspace orthogonal to the row space of $\vZ_{\lo i}$.
That is, it has the same (conditional) distribution as $\vB_i \vu_i$, where:
\begin{enumerate}
  \item $\vB_i$ is a $d \times (n-d+1)$ matrix whose columns form an orthonormal basis for the orthogonal complement of $\vZ_{\lo i}$'s row space;

  \item $\vu_i$ is a uniformly random unit vector in $S^{d-n}$.
\end{enumerate}

As before, for every $i \in [n]$, we define the event $\cB_i$ that $\vK_{\lo i}$ is non-singular and
\begin{align}
          y_i \vy_{\lo i}^\T \vK_{\lo i}^{-1} \vZ_{\lo i} \vLambda \vz_i & \geq 1 .
\end{align}

The Haar measure ensures that the matrices $\vZ$ and $\vZ_{\lo i}$ always have full row rank.
Therefore, because $\vLambda \succ \v0$, the matrices $\vK$ and $\vK_{\lo i}$ are always non-singular.
So we do not need to worry about singularity (c.f.~the events $\cS$ and $\cS_i$).
We accordingly consider the event $\cB := \bigcup_{i=1}^n \cB_i$.
As before, we also define the event $\cE_i(t)$ for every $i \in [n]$ and a given $t > 0$, that
\begin{align}
      \|\vB_i^\T \vLambda \vZ_{\lo i}^\T \vK_{\lo i}^{-1} \vy_{\lo i}\|_2^2 & \geq \frac{d - n + 1}{d} \cdot \frac1t .
\end{align}

By the union bound, we get
\begin{align}
 \Pr(\cB)
  & \leq \sum_{i=1}^n \Pr(\cB_i) \\
  & \leq \sum_{i=1}^n \Pr(\cB_i \mid \cE_i(t)^\compl) + \Pr(\cE_i(t)) ,
  \label{eq:haar-prob-bad-events}
\end{align}
and so we need to upper bound the probabilities $\Pr(\cB_i \mid \cE_i(t)^\compl)$ and $\Pr(\cE_i(t))$ for every $i \in [n]$.

The following lemma upper bounds $\Pr(\cB_i \mid \cE_i(t)^\compl)$, and is analogous to \Cref{lem:prob-B_i} in the proof of \Cref{thm:independent}.
\begin{lemma}\label{lem:prob-B_i_haar}
For any $t > 0$, we have
\begin{align}
    \Pr(\cB_i|\cE_i(t)^\compl) \leq 2 \exp\left(-t\right) .
\end{align}
\end{lemma}

\begin{proof}
First, as discussed above, we have
\begin{align}
  \Pr\del{ y_i\vy_{\lo i}^\T \vK_{\lo i}^{-1} \vZ_{\lo i} \vLambda \vz_i \geq 1}
  & = \Pr\del{\sqrt{d} \cdot y_i\vy_{\lo i}^\T \vK_{\lo i}^{-1} \vZ_{\lo i} \vLambda \vB_i \vu_i \geq 1} \\
  & \leq \Pr\del{ \sqrt{d} \abs[1]{ (\vB_i^\T \vLambda \vZ_{\lo i}^\T \vK_{\lo i}^{-1} \vy_{\lo i})^\T \vu_i } \geq 1 } .
\end{align}
Moreover, $\vu_i$ is independent of $\vZ_{\lo i}$, and as established in \Cref{lem:randunit}, the random vector $\vu_i$ is sub-Gaussian with parameter at most $\mathcal{O}(1/(d - n + 1))$.
Therefore,  $\sqrt{d} \cdot (\vB_i^\T \vLambda \vZ_{\lo i}^\T \vK_{\lo i}^{-1} \vy_{\lo i})^\T \vu_i$ is conditionally sub-Gaussian with parameter at most $\frac{d}{d - n + 1} \cdot \|\vB_i^\T \vLambda \vZ_{\lo i}^\T \vK_{\lo i}^{-1} \vy_{\lo i}\|_2^2 \leq \frac{1}{t}$.
Here, the last inequality follows because we have conditioned on $\cE_i(t)^\compl$.
Therefore, the standard sub-Gaussian tail bound gives us
\begin{align}
  \Pr\del{ \cB_i \mid \cE_i(t)^\compl }
  & \leq 2\exp\del{-t} .  
  \qedhere
\end{align}
\end{proof} 

The next lemma upper bounds $\Pr\del{\cE_i(t)}$ for $t := \tfrac{d-n+1}{d} \cdot \tfrac{d_\infty}{2n}$, and is analogous to \Cref{lem:prob-E_i} in the proof of \Cref{thm:independent}.
\begin{lemma}\label{lem:prob-E_i_haar}
We have
\begin{align}
  \Pr\del{\cE_i\del{\frac{d-n+1}{d} \cdot \frac{d_\infty}{2n}}}
  & \leq \exp\del{-c_1 \cdot d} + 2 \cdot 9^n \cdot \exp\del{-c_2 \cdot \min\{d_2,d_\infty\}}
\end{align}
where $c_1>0$ and $c_2>0$ are universal constants.
\end{lemma}

\begin{proof}
We get
\begin{align}
  \|\vB_i^\T \vLambda \vZ_{\lo i}^\T \vK_{\lo i}^{-1} \vy_{\lo i}\|_2^2
  & \leq \|\vB_i^\T\|_2^2 \cdot \|\vLambda \vZ_{\lo i}^\T \vK_{\lo i}^{-1} \vy_{\lo i}\|_2^2 \\
  & = \|\vLambda \vZ_{\lo i}^\T \vK_{\lo i}^{-1} \vy_{\lo i}\|_2^2 \\
  & \leq  \frac{n\|\vlambda\|_\infty}{\eigmin(\vK_{\lo i})} ,
\end{align}
where we used the fact that $\vB_i$ has orthonormal columns, and the last inequality follows by an identical argument to the proof of \Cref{lem:prob-E_i}.
We will show in particular that 
\begin{align}\label{eq:mineigenvaluehaar}
  \Pr\del{ \eigmin(\vK_{\lo i}) \geq \frac12 \|\vlambda \|_1 }
  & \geq 1-\exp(-c_1 \cdot d) - 2 \cdot 9^{n-1} \cdot \exp(-c_2 \cdot \min\{d_2,d_\infty\}) .
\end{align}
Given \Cref{eq:mineigenvaluehaar}, we can complete the proof of \Cref{lem:prob-E_i_haar}.
This is because we get 
\begin{align}
  \|\vB_i^\T \vLambda \vZ_{\lo i}^\T \vK_{\lo i}^{-1} \vy_{\lo i}\|_2^2
  & \leq \frac{2n\|\vlambda\|_{\infty}}{\|\vlambda\|_{1}}
  = \frac{2n}{d_\infty}
  = \frac{d-n+1}{d} \cdot \frac1t
\end{align}
for
\begin{align}
  t & := \frac{d-n+1}{d} \cdot \frac{d_\infty}{2n} .
\end{align}

We complete the proof by proving \Cref{eq:mineigenvaluehaar}.
Let $\vS \in \bbR^{m \times d}$ be a random matrix with iid standard Gaussian entries with $m := n-1$, and let the singular value decomposition of $\vS$ be $\vS = \vV \vLambda_S \vU^\T$ where $\vV \in \bbR^{m \times m}$ and $\vU \in \bbR^{d \times m}$ are orthonormal matrices.
Then, it is well-known that $\sqrt{d} \cdot \vU^\T$ follows the same distribution as $\vZ_{\lo i}$, and hence $\eigmin(\vK_{\lo i})$ has the same distribution as $d \cdot \eigmin(\vU^T \vLambda \vU)$.
Moreover,
\begin{align}
  d \cdot \eigmin(\vU^\T \vLambda \vU)
  &= \min_{\vv \in \bbR^n, \|\vv\|_2 = 1} \vv^\T \vLambda_S^{-1} \vV^\T \vV \vLambda_S \vU^\T \vLambda \vU \vLambda_S \vV^\T \vV \vLambda_S^{-1} \vv \\
  & \geq \frac{d}{\|\vLambda_S\|_{\op}^2} \cdot \min_{\vv \in \bbR^n, \|\vv\|_2 = 1} \vv^\T \vS \vLambda \vS^\T \vv \\
  & = \frac{d}{\|\vLambda_S\|_{\op}^2} \cdot \eigmin(\vS \vLambda \vS^\T) .
\end{align}
By classical operator norm tail bounds on Gaussian random matrices~\citep[e.g.,][Corollary~5.35]{vershynin2010introduction}, we note that $\|\vLambda_S\|_2^2 \leq \tfrac32 d$ with probability at least $1 - \exp(-c_1 \cdot d)$.
Now, we note that the matrix $\vS \vLambda \vS^\T := \sum_{j=1}^d \lambda_j \vs_j \vs_j^\T$ where the $\vs_j$'s are iid standard Gaussian random vectors in $\bbR^n$. 
So, we directly substitute \Cref{lem:randmatrix} with $\tau := \tfrac14 \|\vlambda\|_1$, and get $\eigmin(\vS \vLambda \vS^\T) \geq \tfrac34 \|\vlambda\|_1$ with probability at least $1 - 2 \cdot 9^m \cdot \exp(-c_2 \cdot \min\{d_2, d_{\infty}\})$.
Putting both of these inequalities together directly gives us \Cref{eq:mineigenvaluehaar} with the desired probability bound, and completes the proof.
\end{proof}

Finally, putting the high probability statements of \Cref{lem:prob-B_i_haar} and~\Cref{lem:prob-E_i_haar} together completes the proof of \Cref{thm:haar}.

\subsection{Proof of \Cref{thm:converse}}

By \Cref{lem:equivalent}, our task is equivalent to lower-bounding the probability that there exists $i \in [n]$ such that $y_i\vy_{\lo{i}}^{\T}\del[0]{\vZ_{\lo{i}}\vZ_{\lo{i}}^{\T}}^{-1}\vZ_{\lo{i}}\vz_{i} \geq 1$.
This event is the union of $n$ (possibly overlapping) events, and hence its probability is at least the probability of one of the events, say, the first one:
\begin{align}
  \Pr\del{ \exists i \in [n] \ \text{s.t.}\ y_i\vy_{\lo i}^\T \vK_{\lo i}^{-1} \vZ_{\lo i} \vz_i \geq 1 }
  & \geq
  \Pr\del{ y_1\vy_{\lo 1}^\T \vK_{\lo 1}^{-1} \vZ_{\lo 1} \vz_1 \geq 1 } .
  \label{eq:iso-main-prob-to-lb1}
\end{align}
Because $\vz_1$ is a standard Gaussian random vector independent of $\vZ_{\lo 1}$, the conditional distribution of $y_1 \vy_{\lo 1}^\T \vK_{\lo 1}^{-1} \vZ_{\lo 1} \vz_1 \mid \vZ_{\lo 1}$ is Gaussian with mean zero and variance $\sigma^2 := \|\vZ_{\lo 1}^\T \vK_{\lo 1}^{-1} \vy_{\lo1 }\|_2^2$.
Therefore, for any $t>0$, we have
\begin{align}
  \Pr\del{ y_1\vy_{\lo 1}^\T \vK_{\lo 1}^{-1} \vZ_{\lo 1} \vz_1 \geq 1 }
  & = \bbE\sbr{ \Pr\del{ \sigma g \geq 1 \mid \sigma } } \quad \text{(where $g \sim \Normal(0,1)$, $g \independent \sigma$)} \\
  & = \bbE\sbr{ \Phi\del{ -1/\sigma } } \\
  & \geq \bbE\sbr{ \Phi\del{ -1/\sigma } \mid \sigma^2 \geq 1/t } \Pr\del{ \sigma^2 \geq 1/t } \\
  & \geq \Phi(-\sqrt{t}) \cdot \Pr(\cE_1(t)) ,
  \label{eq:iso-main-prob-to-lb2}
\end{align}
where $\Phi$ is the standard Gaussian cumulative distribution function, and $\cE_1(t)$ is the event that
\begin{align}
  \sigma^2 = \vy_{\lo 1} \vK_{\lo 1}^{-1} \vZ_{\lo 1} \vZ_{\lo 1}^\T \vK_{\lo 1}^{-1} \vy_{\lo 1} = \vy_{\lo 1} \vK_{\lo 1}^{-1} \vy_{\lo 1} & \geq \frac1t
\end{align}
(as in the proofs of \Cref{thm:independent} and \Cref{thm:haar}).
We now lower-bound the probability of $\cE_1(t)$.
Observe that the $(n-1) \times (n-1)$ random matrix $\vK_{\lo 1} = \vZ_{\lo 1} \vZ_{\lo 1}^\T$ follows a Wishart distribution with identity scale matrix and $d$ degrees-of-freedom.
Moreover, by the rotational symmetry of the standard Gaussian distribution, 
the random variable $\vy_{\lo 1}^\T \vK_{\lo 1}^{-1} \vy_{\lo 1}$ has the same distribution as that of $(\sqrt{n-1} \ve_1)^\T \vK_{\lo 1}^{-1} (\sqrt{n-1} \ve_1) = (n-1) \ve_1^\T \vK_{\lo 1}^{-1} \ve_1$.
It is known that $1 / \ve_1^\T \vK_{\lo 1}^{-1} \ve_1$ follows a $\chi^2$ distribution with $d - (n-2)$ degrees-of-freedom; we denote its cumulative distribution function by $F_{d-n+2}$.
Therefore,
\begin{align}
  \Pr(\cE_1(t)) & = F_{d-n+2}(t(n-1)) .
\end{align}
So, we have shown that
\begin{align}
  \Pr\del{ y_1\vy_{\lo 1}^\T \vK_{\lo 1}^{-1} \vZ_{\lo 1} \vz_1 \geq 1 }
  & \geq \sup_{t\geq0} \Phi(-\sqrt{t}) \cdot F_{d-n+2}(t(n-1)) .
\end{align}
For $t := \tfrac{d-n+4+2\sqrt{d-n+2}}{n-1}$, we obtain $F_{d-n+2}(t) \geq 1-1/e$ by a standard $\chi^2$ tail bound~\citep[Lemma 1]{laurent2000adaptive}.
In this case, we obtain
\begin{align}
  \Pr\del{ y_1\vy_{\lo 1}^\T \vK_{\lo 1}^{-1} \vZ_{\lo 1} \vz_1 \geq 1 }
  & \geq
  \Phi\del{ -\sqrt{\frac{d-n+4+2\sqrt{d-n+2}}{n-1}} } \cdot \del{ 1 - \frac1e }
  \label{eq:converse-final-bound}
\end{align}
as claimed.
\qed

\subsubsection*{Acknowledgements}

This project resulted from a collaboration initiated during the ``Foundations of Deep Learning'' program at the Simons Institute for the Theory of Computing, and we are grateful to the Institute and organizers for their hospitality and support of such collaborative research.
We thank Clayton Sanford for his careful reading and comments on this paper.
DH acknowledges partial support from NSF awards CCF-1740833 and IIS-1815697, a Sloan Research Fellowship, and a Google Faculty Award.
VM acknowledges partial support from a Simons-Berkeley Research Fellowship, support of the ML4Wireless center member companies and NSF grants AST-144078 and ECCS-1343398.
JX was supported by a Cheung-Kong Graduate School of Business Fellowship as a Ph.D.~student at Columbia University during this project.

\bibliographystyle{plainnat}
\bibliography{refs}

\appendix

\section{Anisotropic version of \Cref{thm:converse}}
\label{sec:anisotropic}

Below, we give a version of \Cref{thm:converse} that applies to certain anisotropic settings, depending on some conditions on $\vlambda$.

\begin{theorem} \label{thm:anisotropic-converse}
  There are absolute constants $c>0$ and $c'>0$ such that the following hold.
  Let the training data $(\vx_1,y_1),\dotsc,(\vx_n,y_n)$ follow the model from \Cref{sec:data}, with $\vz_1,\dotsc,\vz_n$ being iid standard Gaussian random vectors in $\bbR^d$, and $y_1,\dotsc,y_n \in \{\pm1\}$ being arbitrary but fixed (i.e., non-random) values.
  Assume $d>n$ and that there exists $k \in \bbN$ and $b>1$ such that $k < (n-1)/c$ and
  \begin{align}
    \frac{\sum_{j=k+1}^d \lambda_j}{\lambda_{k+1}} \leq b(n-1)
  \end{align}
  where $\lambda_1 \geq \lambda_2 \geq \dotsb \geq \lambda_d$.
  Then the probability that at least one training example is not a support vector is at least
  \begin{equation}
    c' \cdot \Phi\del{-\sqrt{\frac{2cb^2(n-1)}{k+1}}} \cdot \del{ 1 - 10e^{-(n-1)/c} } ,
  \end{equation}
  where $\Phi$ is the standard Gaussian cumulative distribution function.
\end{theorem}
Note that the probability bound in \Cref{thm:anisotropic-converse} is at least a positive constant for sufficiently large $n$ provided that the $(k,b)$ obtained as a function of $\vlambda$ satisfy $k+1 \geq c'' b^2(n-1)$ for some absolute constant $c''>0$.

\begin{proof}
  The proof begins in the same way as in that of \Cref{thm:converse}.
  Using the same arguments, we obtain the following lower bound:
  \begin{align}
    \Pr\del{ \exists i \in [n] \ \text{s.t.}\ y_i\vy_{\lo i}^\T \vK_{\lo i}^{-1} \vZ_{\lo i} \vLambda \vz_i \geq 1 }
    & \geq \Pr\del{ y_1\vy_{\lo 1}^\T \vK_{\lo 1}^{-1} \vZ_{\lo 1} \vLambda \vz_1 \geq 1 } \\
    & \geq \Phi(-\sqrt{t}) \cdot \Pr(\cE_1(t))
    \label{eq:main-prob-to-lb}
  \end{align}
  where $\cE_1(t)$ is the event that
  \begin{align}
    \norm{ \vLambda \vZ_{\lo 1}^\T \vK_{\lo 1}^{-1} \vy_{\lo 1} }_2^2 & \geq \frac1t .
  \end{align}

  We next focus on lower-bounding the probability of $\cE_1(t)$.
  (This part is more involved than in the proof of \Cref{thm:converse}.)
  Observe that the (rotationally invariant) distribution of $\vZ_{\lo 1}$ is the same as that of $\vQ\vZ_{\lo 1}$, where $\vQ$ is a uniformly random $(n-1) \times (n-1)$ orthogonal matrix independent of $\vZ_{\lo 1}$.
  Therefore, $\vLambda \vZ_{\lo 1}^\T \vK^{-1} \vy_{\lo 1}$ has the same distribution as
  \begin{align}
    \vLambda (\vQ\vZ_{\lo 1})^\T (\vQ\vZ_{\lo 1}\vLambda\vZ_{\lo 1}^\T\vQ^\T)^{-1} \vy_{\lo 1}
    & = \vLambda \vZ_{\lo 1}^\T \vQ^\T \vQ(\vZ_{\lo 1}\vLambda\vZ_{\lo 1}^\T)^{-1}\vQ^\T \vy_{\lo 1} \\
    & = \sqrt{n-1} \vLambda \vZ_{\lo 1}^\T \vK_{\lo 1}^{-1} \vu
  \end{align}
  where $\vu := \vQ^\T \vy_{\lo 1} / \sqrt{n-1}$ is a uniformly random unit vector, independent of $\vZ_{\lo 1}$.
  Letting $\vM := \vLambda \vZ_{\lo 1}^\T \vK_{\lo 1}^{-1}$, we can thus lower-bound the probability of $\cE_1(t)$ using
  \begin{align}
    \Pr(\cE_1(t))
    & = \Pr\del{ \|\sqrt{n-1} \vM \vu\|_2^2 > 1/t } \\
    & \geq \Pr\del{ \|\sqrt{n-1} \vM \vu\|_2^2 > 1/t \mid \tr{\vM^\T\vM} \geq 2/t } \cdot \Pr\del{ \tr{\vM^\T\vM} \geq 2/t } .
    \label{eq:two-probs-to-lb}
  \end{align}
  We lower-bound each of the probabilities on the right-hand side of \Cref{eq:two-probs-to-lb}.

  We begin with the first probability in \Cref{eq:two-probs-to-lb}, which we handle for arbitrary $t>0$.
  By the Paley-Zygmund inequality, we have
  \begin{align}
    \label{eq:paley-zygmund}
    \Pr\del{ \|\sqrt{n-1} \vM \vu\|_2^2 > \frac12 \bbE\sbr{\|\sqrt{n-1} \vM \vu\|_2^2} \mid \vZ_{\lo 1} }
    & \geq \frac14 \cdot \frac{\bbE\sbr{\|\sqrt{n-1} \vM \vu\|_2^2}^2}{\bbE\sbr{\|\sqrt{n-1} \vM \vu\|_2^4}} .
  \end{align}
  Since $\sqrt{n-1} \vu$ is isotropic, we have
  \begin{align}
    \bbE\sbr{ \|\sqrt{n-1}\vM\vu\|_2^2 \mid \vZ_{\lo 1} } & = (n-1) \tr{\vM^\T\vM \bbE\sbr{ \vu \vu^\T }} = \tr{\vM^\T\vM} .
  \end{align}
  Furthermore, by \Cref{lem:randunit},
  \begin{align}
    \bbE\sbr{ \|\sqrt{n-1}\vM\vu\|_2^4 \mid \vZ_{\lo 1} }
    & \leq C\tr{\vM^\T\vM}^2
  \end{align}
  for some universal constant $C>0$.
  Therefore, plugging back into \Cref{eq:paley-zygmund}, we obtain
  \begin{align}
    \Pr\del{ \|\sqrt{n-1} \vM \vu\|_2^2 > \frac12 \tr{\vM^\T\vM} \mid \vZ_{\lo 1} }
    & \geq \frac14 \cdot \frac{\tr{\vM^\T\vM}^2}{C\tr{\vM^\T\vM}^2}
    = \frac{1}{4C} .
  \end{align}
  Thus we also have the following for arbitrary $t>0$:
  \begin{align}
    \Pr\del{ \|\sqrt{n-1} \vM \vu\|_2^2 > 1/t \mid \tr{\vM^\T\vM} \geq 2/t }
    & \geq \frac1{4C} .
    \label{eq:prob-lb1}
  \end{align}

  We next consider the second probability in \Cref{eq:two-probs-to-lb}, namely $\Pr\del[0]{ \tr{\vM^\T\vM} \geq 2/t }$.
  Recall that we assume there exists $k < (n-1)/c$ and $b>1$ such that
  \begin{align}
    \label{eq:r_k-bound}
    \frac{\sum_{j=k+1}^d \lambda_j}{\lambda_{k+1}} \leq b(n-1) .
  \end{align}
  We claim that for $t := \tfrac{2cb^2(n-1)}{k+1}$,
  \begin{align}
    \Pr\del{ \tr{\vM^\T\vM} \geq \frac2t } & \geq 1 - 10e^{-(n-1)/c} .
    \label{eq:prob-lb2}
  \end{align}
  Indeed, this claim follows from Lemma~16 of \citep{bartlett2020benign}, where their matrix $\vC$ is our matrix $\vM^\T\vM$, except our matrix is $(n-1)\times(n-1)$ instead of $n \times n$, and their matrix $\vSigma$ is our matrix $\vLambda$; see the definitions in their Lemma~8.
  The universal constant $c>0$ in their lemma is the same as ours, and \Cref{eq:r_k-bound} is precisely their condition $r_k(\vSigma) < b(n-1)$ (with the same $k$ and $b$).
  Therefore, the conclusion of their lemma implies, in our notation, that with probability at least $1-10e^{-(n-1)/c}$,
  \begin{align}
    \tr{\vM^\T\vM} & \geq \frac{k+1}{cb^2(n-1)} = \frac2t .
  \end{align}
  This proves the claimed probability bound.

  We conclude from \Cref{eq:main-prob-to-lb}, \Cref{eq:two-probs-to-lb}, \Cref{eq:prob-lb1}, and \Cref{eq:prob-lb2}, that the probability that at least one training example is not a support vector is bounded below by
  \begin{align}
    \Phi\del{ -\sqrt{\tfrac{2cb^2(n-1)}{k+1}} } \cdot \frac1{4C} \cdot \del{ 1 - 10e^{-(n-1)/c} }
  \end{align}
  as claimed.
\end{proof}

\section{Probabilistic inequalities}
\label{sec:prob}

\begin{lemma}
  \label{lemma:net}
  Let $\vM \in \bbR^{n \times n}$ be a symmetric matrix, and let $\cN$ be an $\epsilon$-net of $S^{n-1}$ with respect to the Euclidean metric for some $\epsilon < 1/2$,
  Then
  \begin{align}
    \|\vM\|_2 & \leq \frac1{1-2\epsilon} \max_{\vu \in \cN} |\vu^\T \vM \vu| .
  \end{align}
\end{lemma}
\begin{proof}
  See \citep[Lemma 5.4]{vershynin2010introduction}.
\end{proof}

\begin{lemma}
  \label{lem:randmatrix}
  There is a universal constant $c>0$ such that the following holds.
  Let $\lambda_1, \dotsc, \lambda_d > 0$ be given.
  Let $\vv_1, \dotsc, \vv_d$ be independent random vectors taking values in $\bbR^n$ such that, for some $v>0$,
  \begin{align}
    \bbE(\vv_j) & = \v0 , & \bbE(\vv_j\vv_j^\T) & = \vI_n , & \bbE(\exp(\vu^\T \vv_j)) & \leq \exp(v \|\vu\|_2^2/2) \quad \text{for all $\vu \in \bbR^n$}
  \end{align}
  for all $j=1,\dotsc,d$.
  For any $\tau>0$,
  \begin{align}
    \Pr\del{ \norm{ \sum_{j=1}^d \lambda_j \vv_j\vv_j^\T - \|\vlambda\|_1 \vI_n }_2 \geq \tau }
    & \leq 2 \cdot 9^n \cdot \exp\del{ -c \cdot \min\cbr{ \frac{\tau^2}{v^2 \|\vlambda\|_2^2} ,\, \frac{\tau}{v \|\vlambda\|_\infty} } } .
  \end{align}
  where $\|\vlambda\|_1 := \sum_{j=1}^d \lambda_j$,
  $\|\vlambda\|_2^2 := \sum_{j=1}^d \lambda_j^2$,
  and
  $\|\vlambda\|_\infty := \max_{j \in [d]} \lambda_j$.
\end{lemma}
\begin{proof}
  Let $\cN$ be an $(1/4)$-net of $S^{n-1}$ with respect to the Euclidean metric.
  A standard volume argument of \citet{pisier1999volume} allows a choice of $\cN$ with $|\cN| \leq 9^n$.
  By Lemma~\ref{lemma:net}, we have for any $t>0$,
  \begin{align}
    \Pr\del{ \norm{ \sum_{j=1}^d \lambda_j \vv_j\vv_j^\T - \|\vlambda\|_1 \vI_n }_2 \geq \tau }
    & \leq \Pr\del{ \max_{\vu \in \cN} \abs[4]{ \sum_{j=1}^d \lambda_j (\vu^\T\vv_j)^2 - \|\vlambda\|_1 } \geq \tau/2 } .
  \end{align}
  Next, observe that for any $\vu \in S^{n-1}$, the random variables $\vu^\T \vv_1, \dotsc, \vu^\T \vv_d$ are independent random variables, each with mean-zero, unit variance, and sub-Gaussian with parameter $v$.
  By the Hanson-Wright inequality of~\cite{rudelson2013hanson} and a union bound, there exists a universal constant $c>0$ such that, for any unit vector $\vu \in S^{n-1}$ and any $\tau>0$,
  \begin{align}
    \Pr\del{ \max_{\vu \in \cN} \abs[4]{ \sum_{j=1}^d \lambda_j (\vu^\T\vv_j)^2 - \|\vlambda\|_1 }_2 \geq \tau/2 }
    & \leq 2 \cdot 9^n \cdot \exp\del{ -c \cdot \min\cbr{ \frac{\tau^2}{v^2 \|\vlambda\|_2^2} ,\, \frac{\tau}{v \|\vlambda\|_\infty} } } .
  \end{align}
  The claim follows.
\end{proof}

\begin{lemma}
  \label{lem:randunit}
  Let $\vtheta$ be a uniformly random unit vector in $S^{m-1}$.
  For any unit vector $\vu \in S^{m-1}$, the random variable $\vu^\T \vtheta$ is sub-Gaussian with parameter $v = O(1/m)$.
  Moreover, for any matrix $\vM \in \bbR^{m \times m}$, we have
  \begin{align}
    \bbE\sbr{ \|\vM\vtheta\|_2^4 }
    & \leq \frac{C}{m^2} \tr{\vM^\T\vM}^2
  \end{align}
  where $C>0$ is a universal constant.
\end{lemma}
\begin{proof}
  Let $L$ be a $\chi$ random variable with $m$ degrees-of-freedom, independent of $\vtheta$, so the distribution of $\vz := L\vtheta$ is the standard Gaussian in $\bbR^m$.
  Let $\mu := \bbE[L] = \bbE[L \mid \vtheta] = \sqrt{2} \tfrac{\Gamma((m+1)/2)}{\Gamma(m/2)} = \Omega(\sqrt{m})$.
  By Jensen's inequality, for any $t \in \bbR$,
  \begin{align}
    \bbE\sbr{ \exp(t \vu^\T \vtheta) }
    & = \bbE\sbr{ \exp\del{ \del{\frac{t}\mu \vu}^\T \del{\bbE[ L \mid \vtheta ] \vtheta} } } \\
    & \leq \bbE\sbr{ \exp\del{ \del{\frac{t}{\mu} \vu}^\T \del{ L\vtheta } } } \\
    & = \bbE\sbr{ \exp\del{ \del{\frac{t}{\mu} \vu}^\T \vz } } \\
    & = \exp\del{ \frac{t^2}{2\mu^2} } .
  \end{align}
  It follows that $\vu^\T\vtheta$ is sub-Gaussian with parameter $v = 1/\mu^2 = O(1/m)$.

  Similarly, again by Jensen's inequality,
  \begin{align}
    \mu^4 \cdot \bbE\sbr{ \|\vM\vtheta\|_2^4 }
    & = \bbE\sbr{ \bbE[ L \mid \vtheta ]^4 \|\vM\vtheta\|_2^4 } \\
    & \leq \bbE\sbr{ L^4 \|\vM\vtheta\|_2^4 } \\
    & = \bbE\sbr{ \|\vM\vz\|_2^4 } .
  \end{align}
  Furthermore, a direct computation shows that
  \begin{align}
    \bbE\sbr{ \|\vM\vz\|_2^4 }
    & = 2 \tr{(\vM^\T\vM)^2} + \tr{\vM^\T\vM}^2 \\
    & \leq 3\tr{\vM^\T\vM}^2 .
  \end{align}
  The conclusion follows since $\mu^4 = \Omega(m^2)$.
\end{proof}

\end{document}